\documentclass[reqno]{amsart}
\usepackage[foot]{amsaddr}

\usepackage[left=1in,right=1in,top=1in,bottom=1in]{geometry}
\usepackage{tikz}
\usetikzlibrary{shapes,decorations,calc,arrows}
\usetikzlibrary{decorations.pathreplacing,shapes.geometric}
\usetikzlibrary{calc,positioning}

\usepackage{amsrefs}
\usepackage{amsthm}
\usepackage{amscd}
\usepackage{amsfonts}
\usepackage{amsmath}
\usepackage{amssymb}
\usepackage{mathrsfs}
\usepackage{enumitem}
\usepackage{multirow}
\usepackage{verbatim}
\usepackage{url}
\usepackage{graphicx}
\usepackage{color}

\usepackage{ifthen}

\usepackage{microtype}
\usepackage[hidelinks]{hyperref}

\usepackage[T1]{fontenc} 

\usepackage{tipa}

\newcommand{\N}{\mathbb{N}}
\newcommand{\Z}{\mathbb{Z}}

\newcommand{\C}{\mathbb{C}}

\newcommand{\A}{\mathcal{A}}

\newcommand{\I}{\mathcal{I}}
\newcommand{\cO}{\mathcal{O}}
\newcommand{\G}{\mathcal{G}}
\newcommand{\K}{\mathcal{K}}
\newcommand{\U}{\mathcal{U}}
\newcommand{\M}{\mathcal{M}}

\newcommand{\tr}{\operatorname{tr}}
\newcommand{\Wg}{\operatorname{Wg}}
\newcommand{\tWg}{\overset{\sim}{\operatorname{Wg}}}

\newcommand{\paren}[1]{\left(#1\right)}

\newcommand{\set}[1]{\left\{#1\right\}}

\newcommand{\abs}[1]{\left|#1\right|}

\newtheorem{thm}{Theorem}
\newtheorem{prop}[thm]{Proposition}

\newtheorem{cor}[thm]{Corollary}

\theoremstyle{definition}
\newtheorem{defn}[thm]{Definition}
\newtheorem{ex}[thm]{Example}

\newtheorem{model}{Model}

\definecolor{blueish}{HTML}{3388BB}

\definecolor{reddish}{HTML}{FF4444}

\author{Ian~Charlesworth$^\circ$}
\address{$^\circ$Department of Mathematics, University of California, Berkeley \hfill \url{ilc@math.berkeley.edu}}

\author{Beno\^it~Collins$^\bullet$}
\address{$^\bullet$Department of Mathematics, Graduate School of Science, Kyoto University \hfill \url{collins@math.kyoto-u.ac.jp}}


\title{Matrix models for $\varepsilon$-free independence.}

\begin{document}

\begin{abstract}
	We investigate tensor products of random matrices, and show that independence of entries leads asymptotically to $\varepsilon$-free independence, a mixture of classical and free independence studied by M\l{}otkowski and by Speicher and Wysocza\'nski.
	The particular $\varepsilon$ arising is prescribed by the tensor product structure chosen, and conversely, we show that with suitable choices an arbitrary $\varepsilon$ may be realized in this way.
	As a result we obtain a new proof that $\mathcal{R}^\omega$-embeddability is preserved under graph products of von Neumann algebras, along with an explicit recipe for constructing matrix models.
\end{abstract}

\maketitle

\section{Introduction.}
$\varepsilon$-free probability was originally introduced by M\l{}otkowski in \cite{mlotkowski2004lambda}\footnote{In fact, M\l{}otkowski referred to this concept as ``$\Lambda$-free probability'', but we follow Speicher and Wysocza\'nski here and use the term ``$\varepsilon$-free probability'' instead.} to provide a mixture of classical and free independence; it was shown, for example, that the $q$-deformed Gaussians can be realized as central limit variables in this framework.
Relations of this type were further studied by Speicher and Wysocza\'nski in \cite{speicher2016mixtures}, where cumulants describing $\varepsilon$-freeness were introduced; this led into later work on partial commutation relations in quantum groups studied by Speicher and Weber in \cite{speicher2016quantum}.

The theory of $\varepsilon$-freeness is also connected with the graph products of groups, which has itself been imported into operator algebras and studied, for example, by Caspers-Fima in \cite{caspers2017graph} and by Atkinson in \cite{atkinson2018graph}, where certain stability properties (e.g., the Haagerup property) are shown to be preserved under graph products.
It was also shown by Caspers in \cite{caspers2015} that graph products preserve $\mathcal{R}^\omega$-embeddability.
The connection with our setting is this: algebras are $\varepsilon$-freely independent in their $G$-product when $\varepsilon$ is the adjacency matrix of $G$.

The goal of this short note is to show that $\varepsilon$-independence describes the asymptotic behaviour of tensor products of random matrices; as a consequence, we obtain a method of producing nice matrix models for $\varepsilon$-independent $\mathcal{R}^\omega$-embeddable non-commutative random variables, and another proof that graph products preserve $\mathcal{R}^\omega$-embeddability.

This document is organized as follows.
In Section~\ref{sec:prelim} we present the necessary background on $\varepsilon$-free independence.
In Section~\ref{sec:model} we describe the sort of matrix model we will be considering, and state the main theorem describing the asymptotic behaviour of such models.
In Section~\ref{sec:formula} we compute the necessary moments to obtain the result.
Finally, in Section~\ref{sec:extensions} we discuss some applications of this work to $\mathcal{R}^\omega$-embeddability, and show how the techniques may be adapted to study the case of orthogonally-invariant matrix models instead of unitarily-invariant ones.

\subsection*{Acknowledgements.}

This research project was initiated at PCMI during a Random Matrix program in 2017,  
continued at IPAM during an Operator Algebra program in 2018,
and completed during visits of the authors at each other's institutions in 2019. 
We thank all these institutions for a fruitful working environment. 

Finally, we would like to express our gratitude to Guillaume C\'ebron: after we released the first version of this
manuscript, he pointed out a preprint of Morampudi and Laumann \cite{1808.08674}, which we were not aware of. 
Although the motivations and perspective are rather different --  \cite{1808.08674} is of physical motivation and nature -- 
it turns out to be quite relevant because it studies similar random objects and describes related phenomena to those considered here. 
We hope that the approach through non-commutative probability contained in this manuscript
will trigger further interactions with the study of quantum many-body systems and theoretical physics. 

IC was supported by NSF DMS grant 1803557, and 
BC was supported by JSPS KAKENHI 17K18734, 17H04823, 15KK0162 and ANR- 14-CE25-0003. 

\section{Preliminaries.}
\label{sec:prelim}

We discuss here the notion of $\varepsilon$-freeness, originally introduced by M\l{}otkowski in \cite{mlotkowski2004lambda} and later examined by Speicher and Wysocza\'nski in \cite{speicher2016mixtures}.
This is meant to be a mixture of classical and free independence defined for a family of algebras, where the choice of relation between each pair of algebras is prescribed by a given symmetric matrix $\varepsilon$.
Note that we use here the convention of Speicher and Wysocza\'nski to use $\varepsilon$ rather than M\l{}otkowski's $\Lambda$.

Let $\I$ be a set of indices.
Fix a symmetric matrix $\varepsilon = (\varepsilon_{ij})_{i, j \in \I}$ consisting of $0$'s and $1$'s with $0$'s on the diagonal.
We will think of non-diagonal entries as prescribing free independence by $0$ and classical independence by $1$; the choice of $0$'s on the diagonal was made for later convenience.

\begin{defn}
	Let $\I$ be a set of indices.
	The set $I^\varepsilon_n$ consists of $n$-tuples of indices from $I$ for which neighbours are different modulo the implied commutativity.
	That is, $(i_1, \ldots, i_n) \in I^\varepsilon_n$ if and only if whenever $i_j = i_\ell$ with $1\leq j < \ell \leq n$ there is $k$ with $j < k < \ell$, $i_j \neq i_k$, and $\varepsilon_{i_ji_k} = 0$.
\end{defn}

\begin{defn}
	Let $(\A, \varphi)$ be a non-commutative probability space.
	Then the unital subalgebras $\A_i \subseteq \A$ are $\varepsilon$-independent if $\A_i$ and $\A_j$ commute whenever $\varepsilon_{ij} = 1$, and whenever we take $a_1, \ldots, a_n \in \A$ with $a_j \in \A_{i_j}$, $\varphi(a_j) = 0$, and $(i_1, \ldots, i_n) \in I^\varepsilon_n$ we have $\varphi(a_1\cdots a_n) = 0$.
\end{defn}

For the purpose of introducing matrix models, we introduce the notion of $\varepsilon$-independence for a disjoint collection of 
sets of elements (aka non-commutative random variables) in
$(\A, \varphi)$ and an asymptotic counterpart. 
\begin{defn}\label{def-3}
	Let $(\A, \varphi)$ be a non-commutative probability space and $a_{ij}, (i,j)\in \I\times J$ a family of elements of $\A$.
	They are called $\varepsilon$-independent (or $*$-$\varepsilon$-independent) if the $*$-unital subalgebras
	$\A_i$ generated by $a_{ij}, j\in J$ are $\varepsilon$-independent.
	
	If one takes a sequence of non-commutative probability spaces $(\A_n, \varphi_n)$ and $a_{ij}^{(n)}\in \A_n$ for $(i,j)\in \I\times J$,
	then one defines asymptotic $\varepsilon$-independence as the convergence of $a_{ij}^{(n)}$ in $*$-moments
	to $*$-$\varepsilon$-independent variables.
	
	For an $\varepsilon$-independent family $a_{ij}, (i,j)\in \I\times J$ a family of elements of $\A$, we say that it has a matrix
	model if $\A_n$ can be taken as a matrix algebra. 
\end{defn}

\section{Description of the model.}
\label{sec:model}

\subsection{Sufficient conditions to exhibit a model.}
\label{ss:model}
We are interested in producing a way to model $\varepsilon$-freeness using approximations in finite dimensional matrix algebras.
In particular, we will be interested in models with the following data: 
\begin{enumerate}
	\item for each $N \in \N$, an algebra $\M^{(N)}$ which is the tensor product of a fixed number of copies of $M_N(\C)$, indexed by a non-empty finite set $S$;
	\item for each $\iota \in \I$ a non-empty subset $\K_\iota \subseteq S$, with algebras
		\[\M_\iota^{(N)} := \bigotimes_{s \in \K_\iota}M_N(\C) \otimes \bigotimes_{s \in S \setminus \K_\iota} \C1_N\]
		viewed as included in $\M^{(N)}$ in a way consistent with their indices;
	\item a family of independent unitaries $(U_\iota)_{\iota\in\I}$, with $U_\iota$ Haar-distributed in $\U\paren{\M_\iota^{(N)}}$.
\end{enumerate}
Our goal for the next few sections of this paper will be to establish the following characterisation of the asymptotic behaviour of such models.
\begin{thm}
	\label{thm:main}
	With the notation used above, let $(a_{\iota,j}^{(N)})_N \in U_\iota\M_\iota^{(N)}U_\iota^*$ be a collection of sequences of random matrices, indexed by $(\iota, j) \in \I\times J$, such that for each $\iota$, the collection $\paren{(a_{\iota,j}^{(N)})_N}_j$ has a joint limiting $*$-distribution in the large $N$ limit as per definition $\ref{def-3}$.
	Then the entire family $\paren{(a_{\iota,j}^{(N)})_N}_{\iota,j\in\I\times J}$ has a joint limiting $*$-distribution in the large $N$ limit, and the original families are $\varepsilon$-free where
	\[\varepsilon_{i,j} = \begin{cases}
			1 & \text{ if } \K_i\cap\K_j = \emptyset \\
			0 & \text{ otherwise}
	\end{cases}.\]
\end{thm}
This theorem is proven below as Theorem~\ref{th-asymptotic-eps}(\ref{th-asymptotic-eps-vanishing}).
We will also indicate some ways in which such choices may be made to realise a desired matrix $\varepsilon$.

\begin{model}
	Suppose $\varepsilon$ is prescribed.
	Let $S = \set{\set{i,j} : i, j \in \I \text{ with } \varepsilon_{i,j} = 0}$, and $\K_\iota = \set{s \in S : \iota \in s}$; note that our choice that $\varepsilon_{i,i}=0$ ensures that each $\K_\iota$ is non-empty since $\set{\iota} \in \K_\iota$.
	It is clear that $\K_i\cap\K_j = \emptyset$ precisely when $\varepsilon_{i,j}=1$, as we desire.
\end{model}

\begin{model}
	\label{model:anticliques}
	Suppose $\varepsilon$ is prescribed.
	Let $\G$ be the graph on $\I$ with adjacency matrix $\varepsilon$; that is, indices $i, j \in \I$ are connected in $\G$ if and only if $\varepsilon_{i,j} = 1$; moreover, let $\G^\circ$ be the (self-loop free) complement graph of $\G$, i.e., the graph with vertex set $\I$ where two (distinct) vertices are connected if and only if they are not connected in $\G$.
	Take $S$ to be a clique edge cover of $\G^\circ$: any set of cliques in $\G^\circ$ such that each edge in $\G^\circ$ is contained in at least one element of $S$, and set $\K_\iota = \set{K \in S : \iota\in K}$.
	If we choose $S$ in such a way that each vertex is included in at least one clique, $K_\iota$ will be non-empty.
	Further, if $\varepsilon_{i,j} = 0$ then $(i, j)$ is an edge in $\G^\circ$ and hence contained in at least one clique $K \in S$; then $S \in \K_i \cap \K_j$ which is therefore non-empty.
	It follows that our model will be asymptotically $\varepsilon$-free.

	One may, for example, take $S$ to be the set of all edges and all vertices in $\G^\circ$ (that is, to consist of all cliques of size one or two in $\G^\circ$) which will recover Model A.
	Alternatively, one may take the set of all maximal cliques in $\G^\circ$ (equivalently, the set of all maximal anti-cliques in $\G^\circ$).
	Since $\M^{(N)} \cong M_{N^{\#S}}(\C)$, selecting a smaller set $S$ leads to a much lower rate of dimension growth as $N\to\infty$; unfortunately, finding a minimum clique edge cover is in general very hard \cite{kou1978covering}\footnote{Even approximating a minimum clique edge cover within a factor of $2$ cannot be done in polynomial time unless $P=NP$.}.
\end{model}

\subsection{An example.}\ 
\label{ss:ex}

Let us work out thoroughly an example.
We will take $\varepsilon$ to be the following matrix:
\[\varepsilon = \begin{pmatrix}
		0 & 1 & 0 & 0 & 0 \\
		1 & 0 & 1 & 0 & 0 \\
		0 & 1 & 0 & 1 & 1 \\
		0 & 0 & 1 & 0 & 0 \\
		0 & 0 & 1 & 0 & 0 \\
\end{pmatrix}\]
The corresponding graph and its complement are drawn below.
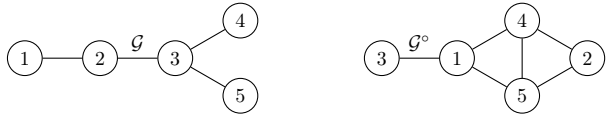
\[\begin{tikzpicture}[baseline, every node/.style={scale=.75}]
		\node (1) [draw, circle] at (0, 0) {1};
		\node (2) [draw, circle] at (1, 0) {2};
		\node (3) [draw, circle] at (2, 0) {3};
		\node (4) [draw, circle] at ($ (3) + (30:1) $) {4};
		\node (5) [draw, circle] at ($ (3) + (-30:1) $) {5};
		\draw (1) -- (2) -- (3) node [midway, above] {$\G$} -- (4);
		\draw (3) -- (5);
	\end{tikzpicture}
	\qquad\qquad
	\begin{tikzpicture}[baseline, every node/.style={scale=.75}]
		\node (3) [draw, circle] at (0, 0) {3};
		\node (1) [draw, circle] at (1, 0) {1};
		\node (4) [draw, circle] at ($ (1) + (30:1) $) {4};
		\node (5) [draw, circle] at ($ (1) + (-30:1) $) {5};
		\node (2) [draw, circle] at ($ (4) + (-30:1) $) {2};
		\draw (3) -- node [midway, above] {$\G^\circ$} (1) -- (4) -- (2) -- (5) -- (1);
		\draw (4) -- (5);
\end{tikzpicture}\]

The maximal anti-cliques in $\G$ are $\set{1, 4, 5}$, $\set{2, 4, 5}$, and $\set{1,3}$.
Therefore if we build Model~\ref{model:anticliques} for this graph using the set of maximal anti-cliques for $S$, we wind up with
\[\begin{array}{rcccccl}
		\M^{(N)} =&	M_N(\C)&\otimes&M_N(\C)&\otimes& M_N(\C);\\
		\M_1^{(N)} =&	M_N(\C)&\otimes&\C&\otimes& M_N(\C);\\
		\M_2^{(N)} =& \C&\otimes& M_N(\C) &\otimes& \C;\\
		\M_3^{(N)} =& \C&\otimes& \C&\otimes& M_N(\C); &\text{ and } \\
		\M_4^{(N)} = \M_5^{(N)} =& M_N(\C)&\otimes& M_N(\C)&\otimes& \C.
\end{array}\]

In particular, if we take $X_i^{(N)}$ to be independent GUE matrices in $\M_i^{(N)}$, we have that $(X_i^{(N)})_i$ converges in law to a family $(S_i)_i$ of semi-circular variables which are $\varepsilon$-free.

It is sometimes useful to visualize $S$ as labelling a set of parallel strings, and elements of $\M_\iota^{(N)}$ as being drawn on the strings corresponding to $\K_\iota$; then, for example, elements commute when they can slide past each other on the strings.
So for example, one could imagine the product $X_1^{(N)}X_3^{(N)}X_5^{(N)}X_2^{(N)}X_1^{(N)}X_4^{(N)}X_2^{(N)}X_4^{(N)}$ as follows:
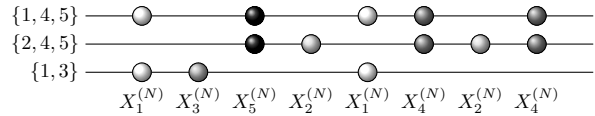
\[\begin{tikzpicture}[baseline, scale=.75, every node/.style={scale=.75}]

		\draw (0,3/2) node [left] {$\set{1,4,5}$} -- (9,3/2);
		\draw (0,2/2) node [left] {$\set{2,4,5}$} -- (9,2/2);
		\draw (0,1/2) node [left] {$\set{1,3}$} -- (9,1/2);

		\def\shades{{0, 0, 15, 30, 50, 100}}
		\def\chords{{"{0}", "{1, 3}", "{2}", "{3}", "{1, 2}", "{1, 2}"}}
		\def\sequence{{0,1,3,5,2,1,4,2,4}}

		\foreach \x in {1, ..., 8} {
			\pgfmathparse{\sequence[\x]}
			\edef\ind{\pgfmathresult}
			\pgfmathparse{\chords[\ind]}
			\edef\chord{\pgfmathresult}
			\pgfmathparse{\shades[\ind]}
			\edef\myshade{\pgfmathresult}
			\foreach \y in \chord {
				\node [draw, shade, circle, ball color=black!\myshade!white] at (\x, 2-0.5*\y) {};
			}
			\node (a) at (\x, 0) {$X_{\ind}^{(N)}$};
		}
\end{tikzpicture}.\]
Note that $(1,3,5,2,1,4,2,4) \in I_8^\varepsilon$ (pictorially, no two columns of variables of the same colour may be slid next to each other along the strings), so in particular asymptotic $\varepsilon$-freeness implies
\[E[\tr(X_1^{(N)}X_3^{(N)}X_5^{(N)}X_2^{(N)}X_1^{(N)}X_4^{(N)}X_2^{(N)}X_4^{(N)})] \xrightarrow{N\to\infty} 0.\]
Here, we took Gaussian unitary ensembles, but it is enough to take matrices that are centred (in expectation)
and unitarily invariant, as we will see further down. 

\section{The formula}
\label{sec:formula}

\subsection{Weingarten integral over tensors}\ 
Let us first fix some notation.
On any matrix algebra, we call $\tr$ the normalized trace. 
For a permutation $\sigma \in S_k$, and a matrix algebra of any dimension $\mathcal{M}_N$, we define
a $k$-linear map $\tr_{\sigma}: \mathcal{M}_N^k\to \mathbb{C}$ by
$$\tr_{\sigma } (A_1\ldots , A_k)=\prod_{cycles}\tr \paren{\prod_{i\in cycle} A_i},$$ 
where
the product in each cycle is taken according to the cyclic order defined by $\sigma$.
For example, $\tr_{(1, 2)(3)}(A_1,A_2,A_3)=\tr(A_1A_2)\tr A_3$. 
Note that by traciality this formula does not depend on the product cycle decomposition.

As in Section~\ref{ss:model}, we fix an index set $\I$ and a finite set $S$, and for $N \in \N$ take $$\M^{(N)} := \bigotimes_{s \in S} M_N(\C).$$
We also fix $\emptyset \neq \K_\iota \subset S$ for each $\iota \in \I$, and set $\M_\iota^{(N)}$ and $(U^{(N)}_\iota)_{\iota\in\I}$ as above, so that $U_\iota^{(N)}$ has Haar distribution on $\U(M_\iota^{(N)})$.

Let us fix $\iota_1, \ldots, \iota_k \in \I$.
The purpose of this section is to compute, for $E$ the the expectation with respect to all $U_\iota^{(N)}$,
\begin{equation}\label{eq-to-compute}
	E(\tr (U_1A_1U_1^*\cdots U_kA_kU_k^*)),
\end{equation}
where $A_j \in M_{\iota_j}^{(N)}$ 
and $U_j = U_{\iota_j}^{(N)}$.

This integral can be computed thanks to the Weingarten formula which we recall here; for further details about Weingarten calculus, we refer the interested reader to \cite{MR2217291,MR1959915}.
The Weingarten calculus says that there exists a central function $\Wg: S_k\times\N\to\mathbb{C}$ such that
$$\int u_{i_1j_1}\ldots u_{i_kj_k}\bar u_{i_1'j_1'}\ldots \bar u_{i_k'j_k'}
=\sum_{\sigma,\tau\in S_k}\delta_{j, \sigma\cdot j'}\delta_{i,\tau\cdot i'}\Wg (\sigma\tau^{-1},N),$$
where the integral is over the Haar measure of the unitary group of $U_N$, and the action of permutations is understood as $\tau\cdot i' = (i_{\tau^{-1}(1)}', \ldots, i_{\tau^{-1}(k)}')$.\footnote{In the literature, it is more common to take the right action of $S_k$; however, using the left action makes some of our arguments later cleaner and does not affect the defining property of the Weingarten function since it is constant on conjugacy classes and so invariant under replacing $\sigma\tau^{-1}$ by $\tau\sigma^{-1}$.}
For the purpose of this paper we just need to know that this function is well defined for $N\ge k$ and
that there exists a function $\mu: S_k\to \mathbb{Z}_*$ so that
$$\Wg (\sigma, N) = \mu (\sigma )N^{-k-|\sigma|} (1+O(N^{-2})).$$
Here, we use the notation $|\sigma |$ for the minimal number of transpositions needed to realize $\sigma$ as a product
of transpositions. This quantity is known to satisfy $|\sigma | = k-\# (\sigma )$ where
$\# (\sigma )$ is the number of cycles of the permutation $\sigma$ in the cycle decomposition (counting
the fixed points too). 
As for $\mu$, although it is irrelevant to this paper, it is closely related to Speicher's non-crossing M\"obius function, and more details can also be found in the above references.

In order to be able to extract asymptotic information when we compute the quantity in \eqref{eq-to-compute}, we introduce
some notation:
\begin{itemize}
	\item $Z=(1\ldots k)$ is the full cycle of the symmetric group $S_k$.
	\item $\tilde S_k$ is the subgroup of permutations $\sigma$ which satisfy $\iota_j = \iota_{\sigma(j)}$ for all $j$, i.e. that 
		stabilize the word $(\iota_1, \ldots, \iota_k)$.
	\item For each $s \in S$, we define $J_s := \set{j | s \in \K_{\iota_j}}$; that is, $J_s$ is those indices where the corresponding label's set contains $s$; we further denote $k_s := \#J_s$. 
		In words, for the string $s$, $J_s$ is the subset of $[[1,k]]$ of indices that act on the leg $s$.
	\item Any $\sigma\in \tilde S_k$ preserves each $J_s$ and so induces a permutation $\sigma_s \in S_{J_s} \cong S_{k_s}$. We will abuse this notation slightly and write $Z_s$ for the full cycle sending each element of $J_s$ to the next one according to the 
		cyclic order.
	\item When we need $\#(\cdot)$ and it is necessary to stress which set a permutation is acting on, we may write something such as $\#_{k_s}(\cdot)$ to mean that the permutation is viewed as an element of $S_{k_s}$ and not as the induced permutation in $S_k$ which is constant off of $J_s$.
\end{itemize}

\begin{ex}
	Let us briefly note how this notation behaves the context of the example in Subsection~\ref{ss:ex}.
	Recall that in this case, we have $k = 8$ and $(\iota_1, \ldots, \iota_8) = (1,3,5,2,1,4,2,4)$.
	Then $\tilde S_k$ is generated by the transpositions $(1\ 5)$, $(4\ 7)$, and $(6\ 8)$.
	Letting $s_1 = \set{1,4,5}, s_2 = \set{2,4,5}$, and $s_3 = \set{1,3}$ (so that $S = \set{s_1, s_2, s_3}$), we have
	\[J_{s_1} = \set{1,3,5,6,8}, \qquad J_{s_2} = \set{3,4,6,7,8}, \qquad J_{s_3} = \set{1,2,5},\]
	which is readily seen from the diagram above.
	We also have
	\[Z_{s_1} = \paren{1\ 3\ 5\ 6\ 8}, \qquad Z_{s_2} = \paren{3\ 4\ 6\ 7\ 8}, \qquad Z_{s_3} = \paren{1\ 2\ 5},\]
	and if $\sigma = (6\ 8)(4\ 7)$ we have
	\[\sigma_{s_1} = (6\ 8), \qquad \sigma_{s_2} = (6\ 8)(4\ 7), \qquad \sigma_{s_3} = id.\]
\end{ex}

Finally, we will introduce a modification of the Weingarten function, specific to our notation.
First, for $\iota \in \I$ we denote by $B_\iota := \set{j : \iota_j = \iota}$ and by $\Pi$ the partition $\set{B_\iota}_{\iota \in \I}$ of $\set{1, \ldots, k}$ (ignoring any empty blocks).
Note that $\tilde S_k$ is the stabilizer of $\Pi$.
We define the function $\tWg$ as follows:
\begin{align*}
	\tWg : \tilde S_k \times \N &\to \C \\
	(\sigma, N) &\mapsto \prod_{\iota \in \I} \Wg(\sigma|_{B_\iota}, N^{\#\K_\iota}).
\end{align*}

We are now able to state our formula:
\begin{thm}\label{th-integration-epsilon}
	The following equation always holds true:
	$$E(\tr (U_1A_1U_1^*\ldots U_kA_kU_k^*) )=
\sum_{\sigma,\tau\in \tilde S_k}\tWg(\sigma^{-1}\tau, N) \tr_{\sigma}(A_1, \ldots , A_k)\prod_{s \in S} N^{\#_{k}(Z^{-1}\tau_s)+\#_{k_s}(\sigma_s )-1}.$$
\end{thm}
We remark that if $k_s > 0$, then $\#(\tau_s^{-1}Z) = \#(\tau_s^{-1}Z_s)$ where the left hand side is viewed in $S_k$ and the right in $S_{k_s}$.

\begin{proof}
	\newcommand{\iijj}{\vec{i},\vec{i'},\vec{j},\vec{j'}}
	\newcommand{\ii}{\vec{i},\vec{i'}}
	\newcommand{\jj}{\vec{j},\vec{j'}}
	Let us use the following vector notation: 
	$\vec{i}=(i_1,\ldots , i_k)$, where $i_1,i_2,\ldots $ are $\#S$-tuples 
	described as
	$i_1= (i_1[1], \ldots , i_1[\#S])$, and similarly for $\vec{i'},\vec{j},\vec{j'}$.
	With this notation, we can expand our product integral:
	$$E(\tr (U_1A_1U_1^*\ldots U_kA_kU_k^*) )= N^{-\#S}\sum_{\vec{i},\vec{i'},\vec{j},\vec{j'}} E (u_{i_1j_1}a_{j_1j_1'} \bar u_{i_1'j_1'}\ldots 
	u_{i_kj_k}a_{j_kj_k'}\bar u_{i_kj_k})\prod_{l=1}^k\delta_{i_l',i_{l+1}},$$
	where the product over $l$ is understood modulo $k$.
	Next, we use the fact that the $a$'s are deterministic while $u$'s from different families are independent:
	\[
		E(\tr (U_1A_1U_1^*\ldots U_kA_kU_k^*) )
		= N^{-\#S}\sum_{\iijj} a_{j_1j_1'}\cdots a_{j_kj_k'} \paren{\prod_{\iota\in\I} E\paren{\prod_{\alpha \in B_\iota} u_{i_\alpha j_\alpha}\bar{u}_{i'_\alpha j'_\alpha}}} \prod_{l = 1}^k \delta_{i_l', i_{l+1}}.
	\]
	We wish to apply the Weingarten formula to each expectation above.
	Since the unitaries in the $\iota$-th term are Haar distributed on $\#\K_\iota$ strings, our resulting dimension is $N^{\#\K_\iota}$; meanwhile the sum will be over permutations of $B_\iota$, and importantly the delta functions involved $\delta_{\cdot, \sigma\cdot(\cdot\cdot)}$ and $\delta_{\cdot, \tau\cdot(\cdot\cdot)}$ will apply only on strings in $\K_\iota$, while the other strings must have equal row and column indices; that is, each term in the product will be of the form
	\[
		E\paren{\prod_{\alpha \in B_\iota} u_{i_\alpha j_\alpha}\bar{u}_{i'_\alpha j'_\alpha}}
		= \sum_{\sigma, \tau\in S_{B_\iota}} \Wg(\sigma\tau^{-1}, N^{\#\K_\iota}) \prod_{\ell\in B_\iota}\prod_{s \in \K_\iota} \delta_{\vec{j}_\ell[s], \vec{j'}_{\sigma^{-1}(\ell)}[s]}\delta_{\vec{i}_\ell[s], \vec{i'}_{\tau^{-1}(\ell)}[s]}\prod_{s \in S\setminus \K_\iota} \delta_{\vec{i}_\ell[s],\vec{j}_\ell[s]}\delta_{\vec{i'}_\ell[s],\vec{j'}_\ell[s]}.
	\]

	Since we are choosing independently permutations on each $B_\iota$, we may instead sum over all $\sigma, \tau \in \tilde S_k$ outside of the product.
	We may then group the Weingarten terms into $\tWg(\sigma\tau^{-1}, N)$.
	We are left with the following:
	\begin{align*}
		&E(\tr (U_1A_1U_1^*\ldots U_kA_kU_k^*) ) \\
		&\qquad\qquad= N^{-\#S}\sum_{\iijj} a_{j_1j_1'}\cdots a_{j_kj_k'}
		\sum_{\sigma, \tau\in \tilde S_k} \tWg(\sigma\tau^{-1}, N)\ \cdot \\
		&\qquad\qquad\qquad\qquad\paren{\prod_{\iota\in\I}\prod_{\ell\in B_\iota}\prod_{s\in\K_\iota}
			\delta_{\vec{j_\ell}[s],\vec{j'}_{\sigma^{-1}(\ell)}[s]}
			\delta_{\vec{i_\ell}[s],\vec{i'}_{\tau^{-1}(\ell)}[s]}
		\prod_{s \in S\setminus \K_\iota} \delta_{\vec{i}_\ell[s],\vec{j}_\ell[s]}\delta_{\vec{i'}_\ell[s],\vec{j'}_\ell[s]}}
		\prod_{l=1}^k\delta_{i'_\ell,i_{\ell+1}}. \\
	\end{align*}

	Let us make some observations to simplify the above expression.
	First, notice that the final product means $\vec{i'}$ is entirely determined by $\vec{i}$, and in fact, $\vec{i'} = Z^{-1}\cdot\vec{i}$.
	The penultimate product means that our choice of $\vec{j}$, $\vec{j'}$ is partially constrained: when choosing $\vec{j}_\ell$, we are free to choose any value on the strings corresponding to $\K_{\iota_\ell}$, but must set $j_\ell[s] = i_\ell[s]$ for $s \notin \K_{\iota_\ell}$.
	Then because we have a factor of $a_{j_\ell, j_\ell'}$ which vanishes as soon as $j_\ell[s]$ and $j_\ell'[s]$ differ for any $s \notin\K_{\iota_\ell}$, we have $j_\ell'[s] = i_\ell[s]$ for $s \notin \K_{\iota_\ell}$ too; that is, a valid choice for $\vec{j}$ is a valid choice for $\vec{j'}$ and vice versa.
	Let us denote by $V_{\vec{i}}$ the set of valid choices of $j$ for a given $\vec{i}$:
	\[V_{\vec{i}} = \set{\vec{j} : \forall\ell,\forall s\in S\setminus\K_{\iota_\ell}, j_\ell[s] = i_\ell[s]}.\]
	When we restrict to summing over this set, we will need to keep the condition that $\vec{i}_\ell[s]=\vec{j}_\ell[s]=\vec{j'}_\ell[s]=\vec{i'}_\ell[s]$; although the first two equalities are satisfied if we restrict $\jj \in V_{\vec{i}}$, the last is not.
	Let us also adopt the following shorthand:
	\[
		\Delta_{\vec{x}, \vec{y}} = \prod_{\iota\in\I}\prod_{\ell\in B_\iota}\prod_{s\in\K_\iota} \delta_{\vec{x_\ell}[s],\vec{y}_{\ell}[s]},
	\]
	Rearranging some terms, we arrive at the expression below:
	\begin{equation}\label{eq:whatamess}
	\begin{split}
		&E(\tr (U_1A_1U_1^*\ldots U_kA_kU_k^*) )
		\\&\qquad=
		N^{-\#S}
		\sum_{\sigma, \tau \in \tilde S_k}\tWg(\sigma\tau^{-1},N)
		\sum_{\vec{i}}\Delta_{\vec{i},Z^{-1}\tau\cdot\vec{i}}\paren{\prod_{\iota\in\I}\prod_{\ell\in B_\iota}\prod_{s\in S\setminus\K_\iota} \delta_{\vec{i}_\ell[s],\vec{i'}_\ell[s]}}
		\sum_{\jj\in V_{\vec{i}}} a_{j_1j_1'}\cdots a_{j_kj'_k}
		\Delta_{\vec{j},\sigma\cdot\vec{j'}}.
		\end{split}
	\end{equation}

	Let us first consider the inner-most sum.
	Note first that each $a_{j_\ell j_\ell'}$ depends only on $j_\ell[s]$ for $s \in \K_\iota$, so this sum does not actually depend on $\vec{i}$.
	Moreover, if we factor the sum based on the cycles of $\sigma$ (both the $a$'s and the delta functions in $\Delta_{\vec{j},\sigma\cdot\vec{j'}}$), we notice that we have exactly a non-normalized trace (over $\M_\iota^{(N)}$ for the corresponding $\iota$) of the $A$'s corresponding to $\sigma$ (since the condition we are enforcing is that on the relevant strings, the column index $j_\ell'[s] = j_{\sigma(\ell)}[s]$ the row index).
	That is, for any $\vec{i}$,
	\[
		\sum_{\jj\in V_{\vec{i}}} a_{j_1j_1'}\cdots a_{j_kj'_k}\Delta_{\vec{j},\sigma\cdot\vec{j'}}
		= \prod_{\iota\in\I} N^{\#_{B_\iota}(\sigma|_{B_\iota})\#\K_\iota}\tr_{\sigma|_{B_\iota}}(A_1,\ldots, A_k)
		= \paren{\prod_{s \in S} N^{\#_{k_s}(\sigma_s)}}\tr_{\sigma^{-1}}(A_1, \ldots, A_k).
	\]

	Next, we turn to the middle sum in \eqref{eq:whatamess}.
	Note that if we look at the conditions along a single string, we find that we need $\vec{i}_\ell[s] = \vec{i}_{\tau(\ell)+1}[s]$ when $s \in \K_{\iota_\ell}$, while $\vec{i}_\ell[s] = \vec{i}_{\ell+1}[s]$ when $s \notin \K_{\iota_\ell}$.
	That is, we need $\vec{i}[s] = Z^{-1}\tau_s\cdot\vec{i}[s]$.
	The number of ways we have of satisfying this condition is precisely $N^{\#(Z^{-1}\tau_s)}$, and so we have
	\[
	\sum_{\vec{i}}\Delta_{\vec{i},Z^{-1}\tau\cdot\vec{i}}\paren{\prod_{\iota\in\I}\prod_{\ell\in B_\iota}\prod_{s\in S\setminus\K_\iota} \delta_{\vec{i}_\ell[s],\vec{i'}_\ell[s]}}
	= \prod_{s\in S} N^{\#_k(Z^{-1}\tau_s)}.
	\]

	Putting everything together, we have
	\[
		E(\tr (U_1A_1U_1^*\ldots U_kA_kU_k^*) )
		= \sum_{\sigma, \tau \in \tilde S_k} \tWg(\sigma\tau^{-1},N) \tr_{\sigma^{-1}}(A_1,\ldots, A_k) \prod_{s\in S} N^{\#_{k_s}(\sigma_s)+ \#_k(Z^{-1}\tau_s) - 1}.
		\]
\end{proof}

Note that some readers might have found it easier to perform the above calculation with graphical calculus
as introduced in \cite{MR2651902}.
As a consequence, we can prove the following:

\begin{thm}\label{th-asymptotic-eps}
	The following two statements hold true:
	\begin{enumerate}
		\item
			\begin{equation}\label{eq-limit-eps}
				E(\tr (U_1A_1U_1^*\ldots U_kA_kU_k^*)) =
				\sum_{\sigma,\tau\in \tilde S_k}\tr_{\sigma}(A_1, \ldots , A_k)
				\mu (\sigma^{-1}\tau)
				\paren{\prod_{s=1}^S N^{|Z_s|-|\sigma_s|-|\sigma_s^{-1}\tau_s|-|\tau_s^{-1}Z_s|}}
				 (1+O(N^{-2})).
			\end{equation}
		\item\label{th-asymptotic-eps-vanishing}
			If $(\iota_1, \dots, \iota_k) \in I_k^\varepsilon$ and $\tr(A_i) = o(1)$ for all $i$, then $E(\tr(U_1A_1U_1^*\cdots U_kA_kU_k^*)) = o(1).$
	\end{enumerate}
\end{thm}

\begin{proof}
	For the first claim, we start with the formula of Theorem \ref{th-integration-epsilon}:
	$$E(\tr (U_1A_1U_1^*\ldots U_kA_kU_k^*) )=
	\sum_{\sigma,\tau\in \tilde S_k}\tWg(\sigma^{-1}\tau, N) \tr_{\sigma}(A_1, \ldots , A_k)\prod_{s \in S} N^{\#(\tau_s^{-1}Z))+\#  (\sigma_s )-1}.$$
	Given that 
	$$\Wg (\sigma, N) = \mu (\sigma ) N^{-k-|\sigma|} (1+O(N^{-2})),$$
	we learn
	\begin{align*}
		\tWg(\sigma, N)
		&= \prod_{\iota\in\I} \mu(\sigma|_{B_\iota})N^{(\#\K_\iota)(-\#B_\iota-\abs{\sigma|_{B_\iota}})}(1+O(N^{-2\#\K_\iota})) \\
		&= \mu(\sigma) \paren{\prod_{\iota\in\I}\prod_{s \in \K_\iota} N^{-\#B_\iota - \abs{\sigma|_{B_\iota}}}}(1+O(N^{-2})) \\
		&= \mu(\sigma) \paren{\prod_{s \in S} N^{-k_s - \abs{\sigma_s}}}(1+O(N^{-2})).\\
	\end{align*}
	Assembling the above, we have
	\[
		E(\tr (U_1A_1U_1^*\ldots U_kA_kU_k^*) )
		= \sum_{\sigma, \tau \in \tilde S_k} \tr_\sigma(A_1, \ldots, A_k) \mu(\sigma^{-1}\tau)\paren{\prod_{s\in S} N^{\#(\tau_s^{-1}Z)+\#(\sigma_s) - 1 - k_s - \abs{\tau_s^{-1}\sigma_s}}}(1+O(N^{-2}));
	\]
	let us turn our attention to the exponent on $N$ in each term of the product.
	We note the following things:
	first, that $\#(\sigma_s) - k_s = -\abs{\sigma_s}$;
	second, that if $k_s = 0$ then $\#(\tau_s^{-1}Z)-1 = 0 = \abs{Z_s} - \abs{\tau_s^{-1}Z_s}$;
	and third, that if $k_s > 0$ then $\#(\tau_s^{-1}Z)-1 = \#(\tau_s^{-1}Z_s) - 1 = k_s-\abs{\tau_s^{-1}Z_s} - (k_s-\abs{Z_s}) = \abs{Z_s} - \abs{\tau_s^{-1}Z_s}$.
	Putting this together, we arrive:
	\[
		E(\tr (U_1A_1U_1^*\ldots U_kA_kU_k^*) )
		= \sum_{\sigma, \tau \in \tilde S_k} \tr_\sigma(A_1, \ldots, A_k) \mu(\sigma^{-1}\tau)\paren{\prod_{s\in S} N^{\abs{Z_s}-\abs{\tau_s^{-1}Z_s}-\abs{\sigma_s^{-1}\tau_s}-\abs{\sigma_s}}}(1+O(N^{-2})).
	\]

	As for the second claim, on the one hand, $\tr_{\sigma}(A_1, \ldots , A_k)$ is always bounded by assumption, which implies that
	the quantities $\tr_{\sigma}(A_1, \ldots , A_k)$ are uniformly bounded in $\sigma , N$. 
	On the other hand, $|Z_s|-|\sigma_s|-|\sigma_s^{-1}\tau_s|-|\tau_s^{-1}Z_s|\le 0$.
	Indeed, this can be reformulated as
	$$|Z_s|\le |\sigma_s|+|\sigma_s^{-1}\tau_s|+|\tau_s^{-1}Z_s|,$$
	which is obvious, since $|\sigma\tau|\le |\sigma |+|\tau |$.
	In addition, it is known that in case of equality, $\sigma_s,\tau_s$ form non-crossing partitions. 
	This is a classical result in combinatorics, we refer to \cite{MR1644993} for one of the first
	uses of this fact in random matrix theory. 

	So, the summands of Theorem \ref{th-integration-epsilon} that are crossing 
	are of order $o(1)$, and we may restrict our sum to non-crossing contributions up to $o(1)$.
	However, according to Proposition \ref{prop-fixed-point}, each such summand involves a $\tr A_i$ so is actually also $o(1)$,
	which proves the vanishing as $N\to\infty$ of the expression $E(\tr (U_1A_1U_1^*\ldots U_kA_kU_k^*))$.
\end{proof}

Note that for the purpose of this proof, we only need to know that $\mu (\sigma_s^{-1}\tau_s)$ is a function, and its
value is irrelevant. It turns out to be an integer, the Biane-Speicher M\"obius function. 
Let us also remark that Equation \eqref{eq-limit-eps} of Theorem \ref{th-asymptotic-eps} is the $\varepsilon$-free
moment cumulant formula of Speicher--Wysocza\'nski \cite{speicher2016mixtures}.

\subsection{A fixed point.}
\begin{prop}\label{prop-fixed-point}
	Suppose that $(\iota_1, \ldots, \iota_k) \in I_k^\varepsilon$.
	Suppose further that $\sigma \in \tilde S_k$ is such that $\sigma_s$ describes a non-crossing partition for each $s \in S$.
	Then $\sigma$ has at least one fixed point.
\end{prop}

\begin{proof}
	Let $B_0 \in \pi_\sigma$ be a block such that $\max B_0 - \min B_0$ is minimal: that is, it is a block of minimal length.
	If $B_0$ is a singleton, it is a fixed point of $\sigma$ and we are done; therefore let us suppose that $B_0$ contains distinct elements $i < j$.
	As $(\iota_1, \ldots, \iota_k) \in I_k^\varepsilon$ and $\iota_i = \iota_j$, there must be some $\ell$ with $i < \ell < j$, $\iota_\ell \neq \iota_i$ and $\varepsilon_{\iota_i, \iota_\ell} = 0$; moreover, there is some $s \in \K_{\iota_i}\cap \K_{\iota_\ell}$.
	Let $B_1\in\pi_\sigma$ be the block containing $\ell$; as $\iota_i \neq \iota_\ell$, $B_0 \neq B_1$.
	Now since $\sigma_s$ describes a non-crossing partition of which both $B_0$ and $B_1$ are blocks, we must have $B_1 \subset \set{i+1, \ldots, j-1}$, contradicting the minimality of the length of $B_0$.
\end{proof}

\section{Further considerations.}
\label{sec:extensions}

\subsection{The orthogonal case.}
The preprint \cite{1808.08674} of Morampudi and Laumann (which we were not aware of when preparing this manuscript) contains a result quite similar to Theorem~\ref{thm:main}, with motivations arising from 
quantum many-body systems, and a graphical approach to their arguments.
The preprint explicitly raises the question of the behaviour of such models with different symmetry groups; we will show that our techniques may be readily adapted to also describe the asymptotic distribution of orthogonally-invariant matrix ensembles.

Let us adopt the setting from Section~\ref{sec:model}, and let $(O_\iota)_{\iota\in\I}$ be a family of independent orthogonal matrices with $O_\iota$ Haar-distributed in $\cO(M_\iota^{(N)})$.
\begin{thm}
	\label{thm:main}
	Let $(a_{\iota,j}^{(N)})_N \in O_\iota\M_\iota^{(N)}O_\iota^*$ be a collection of sequences of random matrices, indexed by $(\iota, j) \in \I\times J$, such that for each $\iota$, the collection $\paren{(a_{\iota,j}^{(N)})_N}_j$ has a joint limiting $*$-distribution in the large $N$ limit as per definition $\ref{def-3}$.
	Then the entire family $\paren{(a_{\iota,j}^{(N)})_N}_{\iota,j\in\I\times J}$ has a joint limiting $*$-distribution in the large $N$ limit, and the original families are $\varepsilon$-free where
	\[\varepsilon_{i,j} = \begin{cases}
			1 & \text{ if } \K_i\cap\K_j = \emptyset \\
			0 & \text{ otherwise}
	\end{cases}.\]
\end{thm}

The thrust of the argument will be to show that the combinatorics are in correspondence to the unitary setting.
We replace the Weingarten function with its orthogonal analogue, but show that the differences caused by this vanish asymptotically and we are left in a situation where the computation may proceed as in the unitary case.

We recall now some useful results from \cite{MR2217291} about the orthogonal Weingarten calculus.
Let $P_{2k}$ be the set of pairings on the set $\set{1, \ldots, 2k}$, and note that each such pairing induces a permutation of order two; then for $N \geq k$, the orthogonal Weingarten function $\Wg_\cO(\cdot, \cdot, N) : P_{2k}\times P_{2k} \to \C$ satisfies
\[\int u_{i_1j_1}\cdots u_{i_{2k}j_{2k}} = \sum_{\sigma, \tau \in P_{2k}} \delta_{j,\sigma\cdot j}\delta_{i, \tau\cdot i} \Wg_\cO(\sigma, \tau, N),\]
where the $u$'s are the entries of a Haar-distributed orthogonal matrix in $\cO_N$.
The asymptotic behaviour of $\Wg_\cO$ is given by
\[\Wg_\cO(\sigma, \tau, N) = \mu(\sigma, \tau) N^{-2k + \frac12\#(\sigma\tau)}(1 + O(N^{-1})),\]
where $\mu : P_{2k}^2 \to \Z_*$ is some function.

We now point out the following: the delta functions arising for any fixed $s \in S$ are independent of all the other strings in $S$.
The delta functions arising corresponding to those indices $2\ell+1, 2\ell+2$ for which $s \in K_{\iota\ell}$ correspond precisely to the delta functions one encounters when computing the expected trace of a corresponding product of matrices in a single matrix family $M_N(\C)$, conjugated by matrices in $\cO(N)$; as in the unitary case, we also have restrictions that the choices of these indices constrain the choices of all indices corresponding to $2\ell+1, 2\ell+2$ for $s \notin K_{\iota_\ell}$.
But conjugation by independent Haar-distributed orthogonal or independent Haar-distributed unitary matrices both lead to asymptotic freeness (see \cite{MR2217291}), i.e.,
\[E(\tr(Q_1B_1Q_1^*\cdots Q_{k_s}B_{k_s}Q_{k_s}^*))
= E(\tr(V_1B_1V_1^*\cdots V_{k_s}B_{k_s}V_{k_s}^*)) + O(N^{-1})\]
whenever the $Q$'s are drawn from a pool of independent Haar-distributed matrices in $\O(N)$, the $V$'s are drawn similarly from $U(N)$ (with the same independence or equality as the choice of $Q$'s), and the $B$'s are arbitrary.
It follows from a linear independence argument that the only pairings which contribute asymptotically to the sum arising from expanding the left hand side are those of the same type as in the unitary case, i.e., those which never match two elements of the same parity.
(A more direct argument may be found in \cite[Lemma 5.1]{MR2217291}.)

It follows that exactly the same computation carries through as in the unitary case, and we arrive at the same result.
We point out that although the $\frac12\#(\sigma\tau)$ in the exponent of the asymptotic expansion of the Weingarten function appears different from the unitary case, it merely accounts for the fact that in this picture every cycle is doubled, depending on whether one begins with an odd or even index.

\subsection{Applications of the main results}

Let us quote two corollaries of our main results.
We start with a corollary of operator algebraic flavor: 

\begin{cor}
Let $K$ be any loop-free and multiplicity free unoriented graph on $k$ vertices, and  $(M_1,\tau_1),\ldots (M_k,\tau_k)$ 
be von Neumann tracial non-commutative probability spaces, i.e. $M_i$ is a finite von Neumann algebra, 
$\tau_i$ is normal faithful tracial. 
If all $(M_i,\tau_i)$ satisfy the Connes embedding property, then so does their von Neumann $\varepsilon$-product
\end{cor}

Note here that we did not define the von Neumann $\varepsilon$-product, it is just the completion of the 
$\varepsilon$-product of $(M_i,\tau_i)$ in the GNS representation associated to the product trace. 

\begin{proof}
It is enough to assume that $M_i$ is finitely generated. The fact that $(M_i,\tau_i)$ satisfy the Connes embedding property
means that
$M_i$ embeds in $R^{\omega}$ in a trace preserving way, and that its generators admit a matrix model.
See for example \cite{MR2465797}.
We use our construction with this matrix model to conclude. 
\end{proof}

Let us finish with a corollary in geometric group theory of the above corollary.
We recall that a group is \emph{hyperlinear} if its group algebra satisfies the Connes embedding problem.
We refer to \cite{MR3408561,MR2460675} for more details on the notions of hyperlinearity (and the Connes problem).
A notion of graph product of group was introduced, cf in particular \cite{MR910401,MR891135,MR880971} for 
early works on this topic. 
We can state the following

\begin{cor}
If $G_1, \ldots , G_k$ are hyperlinear groups and $K$ be any loop-free and multiplicity free unoriented graph on $k$ vertices
then the graph product of these groups over $K$ is also hyperlinear. 
\end{cor}

This is just a consequence on the fact that the group algebra of the graph product is the $\varepsilon$-product of the group 
$*$-algebras of $G_i$, and an application of the above lemma.

\bibliography{thebibliography}

\end{document}